\documentclass{article}

\usepackage[utf8]{inputenc}
\usepackage{amssymb,amsmath,amsthm,amsfonts}
\usepackage{calc}
\usepackage{soul}
\usepackage{xcolor}
\usepackage[rightcaption]{sidecap}
\usepackage{subfigure}
\usepackage{gensymb}
\usepackage{url}
\usepackage{graphicx}
\usepackage{parskip}
\usepackage{fancyhdr}
\usepackage{vmargin}
\usepackage{enumerate} 
\usepackage{wrapfig}
\usepackage{mathrsfs}
\usepackage{float}
\usepackage{authblk}

\newtheorem{teo}{Theorem}[section]
\newtheorem{cor}[teo]{Corollary}
\newtheorem{defi}[teo]{Definition}
\newtheorem*{defi*}{Definition}
\newtheorem{obs}[teo]{Remark}

\newtheorem{prop}[teo]{Proposition}
\newtheorem{lem}[teo]{Lemma}

\title{From Wheel to Jahangir graphs: spectra and energy under structured edge deletion}
\author[1]{Diego Bravo}
\author[2]{Florencia Cubría}
\author[2]{Annabella Zapattini\footnote{The authors were partially supported by PEDECIBA and ANII.}}

\affil[1]{\small\textit{Research \& Development, Cresteo\\
200 E Randolph St.\ Ste 5100, Chicago, IL 60601, USA.}\\
\texttt{diego.bravo@cresteo.com}
}

\affil[2]{\small\textit{Instituto de Matemática y Estadística ``Rafael Laguardia''\\
Facultad de Ingeniería, Universidad de la República\\
Julio Herrera y Reissig 565, Montevideo, Uruguay. 11300}\\
\texttt{fcubria@fing.edu.uy, azapattini@fing.edu.uy}
}

\date{}

\begin{document}
\graphicspath{{./imagenes/}}
\maketitle

\begin{abstract}
The energy of a graph, introduced by Gutman \cite{Gut}, is defined as the sum of the absolute values of the eigenvalues of its adjacency matrix. In this work we study how the energy of the Wheel graph $W_n$ changes when a specific set of spokes is removed. The resulting graphs are precisely the Jahangir graphs $J_{m,n}$, where $m$ divides $n$, and which may also be viewed as the Cycle graph $C_n$ together with an additional vertex adjacent to certain vertices of $C_n$ in a cyclic pattern.

We determine the spectra of $J_{2,n}$, $J_{4,n}$, and $J_{6,n}$ by exploiting a factorization of their characteristic polynomials arising from equitable partitions. This yields explicit closed formulas for their energies. Moreover, {under certain conditions,} we establish the following monotonicity chain:
\[
\mathcal{E}(C_n) \;\leq\; \mathcal{E}(J_{6,n}) \;\leq\; \mathcal{E}(J_{4,n}) 
\;\leq\; \mathcal{E}(J_{2,n}) \;\leq\; \mathcal{E}(W_n).
\]

\noindent\textbf{Keywords:} Graph energy change; Jahangir graphs; wheel-related graphs; energy of graphs; spectrum of graphs.

\noindent\textbf{AMS Subject Classification:} 05C50; 15A18.
\end{abstract}

\section{Introduction}

Throughout this work we consider simple graphs without loops or multiple edges. For a graph $G$ of order $n$, let $A(G)$ denote its adjacency matrix. The characteristic polynomial of $G$ is
\[
\phi(G,x)=\det(xI - A(G)) = x^n + \sum_{i=1}^n a_i x^{n-i}.
\]
The eigenvalues of $A(G)$ will be referred to as the eigenvalues of $G$, and since $A(G)$ is real and symmetric, all of them are real. If $\lambda_1,\dots,\lambda_n$ denote the eigenvalues of $G$, the \emph{energy} of $G$ is defined as
\[
\mathcal{E}(G) = \sum_{i=1}^n |\lambda_i|.
\]
This notion, originating in chemical graph theory through the Hückel Molecular Orbital method, was introduced by Gutman \cite{Gut}. A comprehensive survey of developments up to 2012 is given in \cite{Li}. Section~4.7 of that monograph discusses the area known as \emph{graph energy change}; see, for instance, \cite{Akb, Day, Day2, Gut3, Gut4, Shan, Shan2, Shan3, So}.

In the concluding remarks of \cite{Gut2}, Gutman notes the difficulty of characterizing the graphs $G$ and edges $e$ for which $\mathcal{E}(G-e) \leq \mathcal{E}(G)$. Subsequent work has shown that deleting an edge may increase, decrease, or leave unchanged the energy of a connected graph. In particular, Day \textit{et al.} \cite{Day} provide examples illustrating all three behaviors, and prove that if $e$ is a bridge then $\mathcal{E}(G-e) \leq \mathcal{E}(G)$; hence $\mathcal{E}(T-e) \leq \mathcal{E}(T)$ for any tree $T$. Later, So and Wong \cite{So} proved that for $n \neq 4$, deleting any edge of the cycle $C_n$ decreases its energy. Observe that removing an edge from $C_n$ yields the path graph $P_n$, which allows us to conclude the following inequality $\mathcal{E}(P_n) \leq \mathcal{E}(C_n)$.

Our objective is to analyze the change in energy produced when deleting selected spokes of the Wheel graph $W_n$. We consider several systematic deletion patterns. Removing alternating spokes yields one family, removing blocks of two successive spokes (separated by one remaining spoke) yields another, and similarly for blocks of three consecutive spokes, and so on. These families are precisely the Jahangir graphs $J_{m,n}$, which arise when the hub of $W_n$ is connected to a periodic subset of the vertices of $C_n$.

A key observation is that $J_{m,n}$ inherits part of the spectrum of $C_n$. For $m$ even, we use equitable partitions to obtain a factorization of $\phi(J_{m,n},x)$ that isolates the inherited eigenvalues and allows us to compute the remaining ones explicitly. This enables us to determine the full spectra of $J_{2,n}$, $J_{4,n}$, and $J_{6,n}$, and then compute and compare their energies with those of $C_n$ and $W_n$.

\section{Preliminaries}

In this section we introduce the family of Jahangir graphs and recall the spectra and energies of the classical graphs $C_n$, $P_n$, and $W_n$.\\
Since the adjacency matrix of the cycle graph $C_n$ is circulant, its spectrum can be described explicitly using the general theory of circulant matrices. We recall the relevant result below.

\begin{prop}[\cite{Gra}]\label{prop:matrizCirculante}
Let $A$ be a circulant matrix. In particular, this applies to the adjacency matrix $A(C_n)$ of the cycle graph. Then the set of eigenvectors of $A$ is
\[
\mathcal{Y}=\{y_0, y_1, \dots, y_{n-1}\},
\qquad
y_k = (1, \rho_k, \rho_k^2, \dots, \rho_k^{\,n-1}),
\]
where $\rho_k=e^{\frac{2\pi k}{n}i}$, for $k=0,\dots,n-1$.  
Moreover, each eigenvector $y_k$ corresponds to the eigenvalue
\[
\lambda_k=\sum_{j=0}^{n-1}c_j e^{\frac{2\pi jk}{n}i},
\]
where $(c_0, c_1, \dots, c_{n-1})$ is the first row of $A$.
\end{prop}

Since one of our main objectives is to compare the energy of $W_n$ with that of the related Jahangir graphs, we recall the definition of the Wheel graph.

\begin{defi}
For $n \ge 3$, the \textbf{Wheel graph} $W_n$ is obtained by adding a new vertex, called the \textbf{center}, to the cycle graph $C_n$, and joining it to every vertex of $C_n$.
\end{defi}

The next recurrence is an immediate consequence of Theorem 2.2.1 in \cite{Cve}.

\begin{prop}\label{prop:recurrencia}
For all $n \ge 2$,
\[
\phi(P_n,x) = x\,\phi(P_{n-1},x) - \phi(P_{n-2},x),
\qquad
\phi(P_0,x) = 1.
\]
\end{prop}

\begin{prop}\label{prop:espectros}
We now recall the spectra of the fundamental graphs $C_n$, $P_n$, and $W_n$.\\

\begin{enumerate}
    \item The spectrum of the cycle graph is
    \[
    Sp(C_n) = \bigl[\, 2\cos(2\pi k/n): k=0,\dots,n-1 \,\bigr].
    \]

    \item The spectrum of the path graph is
    \[
    Sp(P_n) = \bigl[\, 2\cos(\pi k/(n+1)): k=1,\dots,n \,\bigr].
    \]

    \item The spectrum of the wheel graph is
    \[
    Sp(W_n)
    =
    \bigl(Sp(C_n)-[2]\bigr) \;\sqcup\; \bigl[\,1+\sqrt{n+1},\; 1-\sqrt{n+1}\,\bigr],
    \]
    where $-$ and $\sqcup$ denote multiset difference and multiset union, respectively.
\end{enumerate}
\end{prop}

\begin{proof}
\leavevmode\\[1pt]
\begin{enumerate}
    \item See Example 1.1.4 in \cite{Cve}.

    \item See Section 2.6, item 7, in \cite{Cve}.

    \item Since $W_n = K_1 \vee C_n$ is the join of $K_1$ and $C_n$, Theorem 2.1.8 of \cite{Cve} yields
    \[
    \phi(W_n,x)
    =
    \frac{\phi(K_1,x)\,\phi(C_n,x)}{x(x-2)} \,\bigl( x(x-2) - n \bigr)
    =
    \frac{\phi(C_n,x)}{x-2}\,(x^2 - 2x - n).
    \]
    The spectrum of $W_n$ consists of all eigenvalues of $C_n$ except $2$, together with the roots of $x^2 - 2x - n$, which are $1\pm\sqrt{n+1}$.
\end{enumerate}
\end{proof}

Throughout this work the notation 
\[
\lambda_k^{(n)} = 2\cos\!\left(\frac{2\pi k}{n}\right), \qquad k = 0,\dots,n-1,
\]
denotes the eigenvalues of the cycle $C_n$.

The corresponding energies are also well known, and their closed forms naturally split into three cases, since every positive integer $n$ is either odd, a multiple of $4$, or satisfies that $n-2$ is a multiple of $4$.

\begin{prop}[\cite{Li}]
The energy of the cycle is
\[
\mathcal{E}(C_n) =
\begin{cases}
\displaystyle \frac{4\cos(\pi/n)}{\sin(\pi/n)} & \text{if } n\equiv 0 \pmod 4, \\[6pt]
\displaystyle \frac{4}{\sin(\pi/n)} & \text{if } n\equiv 2 \pmod 4, \\[6pt]
\displaystyle \frac{2}{\sin(\pi/(2n))} & \text{if } n\equiv 1 \pmod 2.
\end{cases}
\]

\end{prop}

\begin{obs}
Using the spectrum in Proposition~\ref{prop:espectros}, the energy of the wheel graph satisfies
\[
\mathcal{E}(W_n)
=
\mathcal{E}(C_n)
+
\bigl|1+\sqrt{n+1}\bigr|
+
\bigl|1-\sqrt{n+1}\bigr|
-
|2|.
\]
\end{obs}

The previous Remark directly yields the following expression.

\begin{prop}\label{prop:energiaCicloRueda}

The energy of the wheel graph is
\[
\mathcal{E}(W_n) =
\begin{cases}
\displaystyle \frac{4\cos(\pi/n)}{\sin(\pi/n)} + 2(\sqrt{n+1}-1), & n\equiv 0 \pmod 4, \\[6pt]
\displaystyle \frac{4}{\sin(\pi/n)} + 2(\sqrt{n+1}-1), & n\equiv 2 \pmod 4, \\[6pt]
\displaystyle \frac{2}{\sin(\pi/(2n))} + 2(\sqrt{n+1}-1), & n\equiv 1 \pmod 2.
\end{cases}
\]
\end{prop}

We conclude this section with the definition of the Jahangir graphs.

\begin{defi}\label{jahangir}
Let $m$ and $n$ be positive integers with $n \ge 3$ and $m\mid n$.  
The \textbf{Jahangir graph} $J_{m,n}$ is obtained by adding a new vertex, called the \textbf{center}, to the cycle $C_n$, and making it adjacent to exactly $n/m$ vertices of $C_n$, in such a way that within each block of $m$ consecutive vertices, precisely one is adjacent to the center.
\end{defi}

If $m=1$ then $J_{1,n}=W_n$, so the wheel graph is the first member of the Jahangir family.

The graphs $J_{2,6}$ and $J_{3,9}$ are shown in Figure~\ref{fig:Jahangir}. Note that $J_{2,6}$ is bipartite, whereas $J_{3,9}$ is not.

\begin{figure}[h]
    \centering
    \includegraphics[scale=1.3]{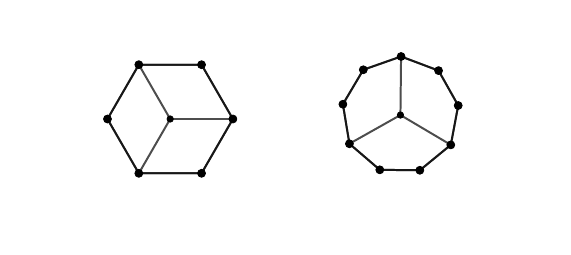}
    \caption{The graphs $J_{2,6}$ and $J_{3,9}$.}
    \label{fig:Jahangir}
\end{figure}

\section{Common eigenvalues of $J_{m,n}$ and $C_n$}
\label{sec:common}

In this section we identify the eigenvalues that the Jahangir graph $J_{m,n}$ inherits from the cycle $C_n$. 
We begin with a structural observation: the bipartiteness of $J_{m,n}$ depends exactly on the parity of $m$.

\begin{lem}
The graph $J_{m,n}$ is bipartite if and only if $m$ is even.
\end{lem}

\begin{proof}
A graph is bipartite if and only if it contains no odd cycle.  
The cycles of $J_{m,n}$ have lengths $n$ and $km+2$ for $k = 1,\dots,\frac{n}{m} - 1$.  
Thus $J_{m,n}$ has no odd cycles precisely when $m$ is even.
\end{proof}

Before stating the main results of this section, recall that the vectors $y_k$ defined in Proposition~\ref{prop:matrizCirculante} form an eigenbasis for the adjacency matrix of $C_n$.  
The next lemma isolates a key summation that determines which of these eigenvectors extend to eigenvectors of $J_{m,n}$.

\begin{lem}\label{lem:sumatoria}
Let $y_k \in \mathcal{Y}$ be as in Proposition~\ref{prop:matrizCirculante}. Then
\[
\sum_{\substack{0\le i \le n-1\\ i \equiv 0 \,(\mathrm{mod}\ m)}} (\rho_k)^i
=
\begin{cases}
\displaystyle \frac{n}{m}, & \text{if $k$ is a multiple of }\tfrac{n}{m},\\[4pt]
0, & \text{otherwise}.
\end{cases}
\]
\end{lem}

\begin{proof}
We compute:
\[
\sum_{\substack{0\le i\le n-1\\ i\equiv 0 \,(\mathrm{mod}\ m)}} (\rho_k)^i
=
\sum_{j=0}^{\frac{n}{m}-1} (\rho_k)^{mj}
=
\sum_{j=0}^{\frac{n}{m}-1} (\rho_k^{\,m})^j.
\]

If $\rho_k^m = 1$, then every term equals 1 and the sum is $n/m$.  
Otherwise,
\[
\sum_{j=0}^{\frac{n}{m}-1} (\rho_k^{m})^j
=
\frac{\rho_k^{n}-1}{\rho_k^{m}-1}
=
0,
\]
since $\rho_k$ is an $n$th root of unity.  

Finally,
\[
\rho_k^{m}=1
\quad\Longleftrightarrow\quad
e^{2\pi i k m/n}=1
\quad\Longleftrightarrow\quad
k \text{ is a multiple of } \tfrac{n}{m}.
\]
\end{proof}

The next observation records the block structure of $A(J_{m,n})$.

\begin{obs}
Since $C_n$ is an induced subgraph of $J_{m,n}$, the adjacency matrix of $J_{m,n}$ has the form
\[
A(J_{m,n})
=
\begin{pmatrix}
A(C_n) & r\\
r^{t} & 0
\end{pmatrix},
\]
where the vector $r$ encodes the neighbors of the center.
\end{obs}

We now determine which eigenvectors of $C_n$ remain eigenvectors of $J_{m,n}$. 

\begin{prop}\label{prop:common-no-multiple}
Let $m$ divide $n$, and let $1 \le k \le n-1$ be an integer that is not a multiple of $n/m$. 
Then
\[
x_k = \begin{pmatrix} y_k \\ 0 \end{pmatrix}
\]
is an eigenvector of $J_{m,n}$ with eigenvalue $\lambda_k^{(n)}$.
\end{prop}

\begin{proof}
Using the block form of $A(J_{m,n})$,
\[
A(J_{m,n}) x_k
=
\begin{pmatrix}
A(C_n) y_k\\[4pt]
\displaystyle \sum_{\substack{0\le j\le n-1\\ j\equiv 0\,(\mathrm{mod}\ m)}} \rho_k^{\,j}
\end{pmatrix}
=
\begin{pmatrix}
\lambda_k^{(n)} y_k \\[4pt]
0
\end{pmatrix},
\]
where the lower entry vanishes by Lemma~\ref{lem:sumatoria}.  
Thus $A(J_{m,n})x_k = \lambda_k^{(n)} x_k$.
\end{proof}

\begin{cor}\label{cor:vapsCnJmn}
If $1 \le k \le n-1$ and $k$ is not a multiple of $n/m$, then
\[
\lambda_k^{(n)} = 2\cos\!\left(\frac{2\pi k}{n}\right)
\]
is also an eigenvalue of $J_{m,n}$ (with the same multiplicity as in $C_n$).
\end{cor}

The previous result captures all eigenvalues of $C_n$ except those coming from indices $k$ that are multiples of $n/m$; these correspond to the spectrum of $C_m$.

\begin{obs}\label{obs:espectroCnmenosCm}
If $k = i \cdot \frac{n}{m}$ for $i = 0,\dots,m-1$, then
\[
\lambda_k^{(n)} = 2\cos\!\left(\tfrac{2\pi i}{m}\right) = \lambda_i^{(m)},
\]
the eigenvalues of $C_m$.  
Thus the multisets satisfy
\[
Sp(C_n) - Sp(C_m) \ \sqsubset\ Sp(J_{m,n}),
\]
and hence
\[
\frac{\phi(C_n,x)}{\phi(C_m,x)}
\text{ divides }
\phi(J_{m,n},x).
\]
\end{obs}

\begin{obs}\label{obs:mIgualDos}
When $m=2$, we have {$\lambda_k^{(n)} = 2\cos\!\left(\frac{2\pi k}{n}\right)$ is an eigenvalue of $J_{m,n}$ with the same multiplicity as in $C_n$ if $k$ is not a multiple of $\frac{n}{2}$, thus, } 
\[
Sp(C_n) - [-2,2] \ \sqsubset\ Sp(J_{2,n}),
\]
and therefore
\[
\frac{\phi(C_n,x)}{(x-2)(x+2)}
\mid
\phi(J_{2,n},x).
\]
%Since $n$ is even, each of these eigenvalues appears in $J_{2,n}$ with multiplicity $2$.
\end{obs}

\begin{obs}\label{obs:cantidadVapsComunesJmnCn}
Corollary~\ref{cor:vapsCnJmn} shows that $J_{m,n}$ and $C_n$ share at least $n-m$ eigenvalues (counted with multiplicity).
\end{obs}

We next show that essentially all eigenvalues of $C_n$ appear in $J_{m,n}$, although some may lose multiplicity.  
To that end we recall the notion of interlacing, following \cite{Hae}.

Let $\lambda_1 \ge \cdots \ge \lambda_n$ and $\mu_1 \ge \cdots \ge \mu_m$ be real sequences with $m<n$.  
We say that $(\mu_i)$ \textbf{interlaces} $(\lambda_i)$ if
\[
\lambda_i \ge \mu_i \ge \lambda_{n-m+i},
\qquad i=1,\dots,m.
\]

Corollary 2.2 of \cite{Hae} states that if $B$ is a principal submatrix of a symmetric matrix~$A$, then the eigenvalues of $B$ interlace those of $A$.

\begin{prop}\label{prop:interlacing}
Every eigenvalue of $C_n$ other than $2$ (and also $-2$ when $m$ is even) is an eigenvalue of $J_{m,n}$.
\end{prop}

\begin{proof}
Since $A(C_n)$ is a principal submatrix of $A(J_{m,n})$, their eigenvalues interlace.  
For any eigenvalue $\lambda$ of $C_n$ with $\lambda \ne 2$ (and $\lambda \ne -2$ if $m$ is even), the multiplicity $mg(\lambda)$ in $C_n$ is $2$.  
Thus there must exist an eigenvalue $\mu$ of $J_{m,n}$ such that
\[
\lambda \le \mu \le \lambda,
\]
forcing $\mu = \lambda$.
\end{proof}

\begin{obs}
For $m=2$, Proposition~\ref{prop:interlacing} gives no additional information beyond what was already obtained in Remark~\ref{obs:mIgualDos}.
\end{obs}

\begin{cor}\label{cor:vapsDeCmEnJmn}
Let $m\ge 3$.  
Since $Sp(C_m)\sqsubset Sp(C_n)$, Proposition~\ref{prop:interlacing} implies that the eigenvalues of $C_m$ different from $2$ (and $-2$ if $m$ is even) also appear in $Sp(J_{m,n})$.  
Explicitly,
\[
\left[\lambda_i^{(m)} : i = 1,\dots,\tfrac{m-2}{2}\right]
\quad\text{if $m$ is even},
\qquad
\left[\lambda_i^{(m)} : i = 1,\dots,\tfrac{m-1}{2}\right]
\quad\text{if $m$ is odd}.
\]
\end{cor}

\begin{obs}
The number of common eigenvalues of $J_{m,n}$ and $C_m$ is therefore at least $\tfrac{m-2}{2}$ when $m$ is even and $\tfrac{m-1}{2}$ when $m$ is odd.
\end{obs}

\section{Factorization of $\phi(J_{m,n},x)$ for $m$ even}
\label{sec:factorization}

In this section we obtain a closed factorization of the characteristic polynomial of the Jahangir 
graph $J_{m,n}$ when $m$ is even.  
The key tool is the theory of equitable partitions, which produces a quotient matrix whose eigenvalues 
lift to eigenvalues of the original graph.  
All remaining eigenvalues of $J_{m,n}$, those not inherited from $C_n$ or $C_m$, will arise as roots of this quotient polynomial.

\subsection{Equitable partitions and quotient matrices}
Let $G$ be a graph and $\pi=\{S_1,\dots,S_k\}$ a partition of $V(G)$.  
We call $\pi$ \emph{equitable} if each vertex in $S_i$ has exactly $b_{ij}$ neighbours in $S_j$, 
independently of the chosen vertex.  
The matrix $B=(b_{ij})$ is the adjacency matrix of the \emph{quotient multigraph} $G/\pi$.  

The following is a special case of Theorem~9.3.3 of \cite{God}.

\begin{prop}\label{prop:equitable}
If $\pi$ is an equitable partition of $G$, then the characteristic polynomial 
$\pi_G(x)=\det(xI - A(G/\pi))$ divides $\phi(G,x)$, and every eigenvalue of $A(G/\pi)$ is an 
eigenvalue of $A(G)$.  
Moreover, each eigenvector of $A(G/\pi)$ lifts to an eigenvector of $A(G)$.
\end{prop}

{Thus, once an equitable partition is identified, the factorization of the characteristic polynomial of the quotient matrix can be used in the factorization of
$\phi(G,x)$.}

\subsection{An equitable partition of the Jahangir graph}
Fix integers $m,n$ with $m$ even and $m\mid n$.  
Write $J_{m,n}$ in the usual way as a cycle of length $n$ together with a center vertex adjacent to ${n}/{m}$ 
equally spaced vertices.

The natural geometry of $J_{m,n}$ yields an equitable partition
\[
\pi = \{S_1,S_2,\dots,S_{\frac{m}{2}+2}\},
\]
defined as follows:
\begin{itemize}
    \item $S_1$ is the center vertex.
    \item $S_2$ consists of the $n/m$ neighbours of the center.
    \item $S_3$ contains the neighbours of $S_2$ in the cycle, not in $S_1$.
    \item $S_i=\left\{x\in V(J_{m,n}):x\sim y,\ \text{with}\  y\in S_{i-1}\ \text{and}\  {x}\not\in S_{i-2} \right\}$, for $4\le i\le \frac{m}{2}+2$.
\end{itemize}

{Each class $S_i$ consists of exactly those vertices in $C_n$ whose distance from the nearest spoke is equal to $i-2$.}

The following proposition records the adjacency matrix of the quotient graph.

\begin{prop}\label{prop:quotient-even}
Let $m$ be even.  
Then the quotient matrix $A(J_{m,n}/\pi)$ is the $(\frac{m}{2}+2)\times(\frac{m}{2}+2)$ matrix
\[
Q_{m,n}=
\begin{pmatrix}
0 & \frac{n}{m} & 0 & 0 & \cdots & 0 & 0 & 0\\
1 & 0 & 2 & 0 & \cdots & 0 & 0 & 0\\
0 & 1 & 0 & 1 & \cdots & 0 & 0 & 0\\
0 & 0 & 1 & 0 & \cdots & 0 & 0 & 0\\
\vdots & \vdots & \vdots & \vdots & \ddots & \vdots & \vdots & \vdots\\
0 & 0 & 0 & 0 & \cdots & 0 & 1 & 0\\
0 & 0 & 0 & 0 & \cdots & 1 & 0 & 1\\
0 & 0 & 0 & 0 & \cdots & 0 & 2 & 0
\end{pmatrix}.
\]
\end{prop}

\begin{proof}
Each row corresponds to the number of neighbours a vertex in $S_i$ has in each class $S_j$.  
The symmetry of the cycle, along with the regular spacing of vertices adjacent to the center, 
ensures that the neighbour counts are constant within each class.  
The counted values match exactly the entries displayed above.
\end{proof}

Denote by 
\[
\pi_{J_{m,n}}(x)=\det\bigl(xI - Q_{m,n}\bigr)
\]
the characteristic polynomial of the quotient matrix.

\subsection{A determinant identity}
To compute $\pi_{J_{m,n}}(x)$, we require a recurrence for a truncated path-like matrix arising from the equitable partition.

\begin{lem}\label{lem:Mn}
Let $M_l$ be the $l\times l$ matrix
\[
M_l=
\begin{pmatrix}
0 & 1 & 0 & \cdots & 0 & 0 & 0\\
1 & 0 & 1 & \cdots & 0 & 0 & 0\\
0 & 1 & 0 & \cdots & 0 & 0 & 0\\
\vdots & \vdots & \vdots & \ddots & \vdots & \vdots & \vdots\\
0 & 0 & 0 & \cdots & 0 & 1 & 0\\
0 & 0 & 0 & \cdots & 1 & 0 & 1\\
0 & 0 & 0 & \cdots & 0 & 2 & 0
\end{pmatrix}.
\]
Then
\[
\det(xI - M_l)
=
\phi(P_l,x)-\phi(P_{l-2},x),
\qquad l\ge 2,
\]
with the convention $\phi(P_0,x)=1$.
\end{lem}

\begin{proof}
A direct Laplace expansion along the last row gives the desired difference of path polynomials.
\end{proof}

\subsection{Computation of the quotient polynomial}
We now compute $\pi_{J_{m,n}}(x)$.

\begin{prop}\label{prop:quotient-polynomial}
Let $m$ be even. Then
\[
\pi_{J_{m,n}}(x)
=
\begin{cases}
\bigl((x^2-\frac{n}{4})x -2x\bigr)\,\phi(P_1,x)
   -2\bigl(x^2-\frac{n}{4}\bigr)\phi(P_0,x), & m=4,\\[6pt]
\bigl(x^3-(\frac{n}{m}+4)x\bigr)\,
     \phi(P_{\frac{m}{2}-1},x)
    +2\frac{n}{m}\,\phi(P_{\frac{m}{2}-2},x), & m\ge 6.
\end{cases}
\]
\end{prop}

\begin{proof}
For $m=4$, the quotient matrix is $4\times4$ and a direct expansion together with 
Lemma~\ref{lem:Mn} yields the stated expression.  

For $m\ge 6$, the matrix $Q_{m,n}$ contains a centered $M_l$ block of size $l=\frac{m}{2}$.  
Laplace expansion along the first two rows produces the linear combination of 
$\phi(P_{l},x)$, $\phi(P_{l-1},x)$, and $\phi(P_{l-2},x)$ recorded above.  
Applying the recurrence $\phi(P_l,x)=x\phi(P_{l-1},x)-\phi(P_{l-2},x)$ 
gives the desired formula.
\end{proof}

\subsection{The factorization Theorem for even {$m$}}\label{subsec:factor-even}
We now identify which eigenvalues of $C_n$ and $C_m$ appear in $J_{m,n}$ and isolate the remaining ones.

For completeness, we record the factorization of the characteristic polynomial of $J_{2,n}$.
When $m=2$, the center vertex is adjacent to $n/2$ equally spaced vertices of the cycle.
This yields an equitable partition with three classes:
\[
S_1=\{\text{center}\}, \quad
S_2=\{\text{vertices adjacent to the center}\}, \quad
S_3=\{\text{remaining $n/2$ cycle vertices}\}.
\]
With respect to this partition, the quotient matrix is
\[
Q_{2,n}
=
\begin{pmatrix}
0 & \tfrac{n}{2} & 0\\[2pt]
1 & 0 & 2\\[2pt]
0 & 2 & 0
\end{pmatrix}.
\]

Its characteristic polynomial is
\[
\pi_{J_{2,n}}(x)
=
\det(xI - Q_{2,n})
= x^{3} - \Bigl(4+\tfrac{n}{2}\Bigr)x,
\]
whose roots are
\[
\left[\sqrt{\tfrac{n}{2}+4},\,0,\,-\sqrt{\tfrac{n}{2}+4}\right].
\]

Since the only eigenvalues of $C_n$ that do not appear in $J_{2,n}$ are $\pm 2$,  
the characteristic polynomial of $J_{2,n}$ factorizes as
\[
\phi(J_{2,n},x)
=
\frac{\phi(C_n,x)}{(x-2)(x+2)} \cdot 
 \pi_{J_{2,n}}(x).
\]

This provides a complete description of the eigenvalues of $J_{2,n}$ and 
serves as the base case for the factorizations obtained for larger even values of $m$.

\begin{prop}\label{prop:remaining-eigenvalues}
Let $m\ge 4$ be even.  
Then the roots of $\pi_{J_{m,n}}(x)$ are precisely the $\frac{m}{2}+2$ eigenvalues of $J_{m,n}$ 
that are not obtained from:
\begin{enumerate}
    \item eigenvalues $\lambda_k^{(n)}$ of $C_n$ with $k$ not divisible by $n/m$, and
    \item eigenvalues $\lambda_i^{(m)}$ of $C_m$ with $1\le i\le \frac{m-2}{2}$.
\end{enumerate}
\end{prop}

\begin{proof} 

%\begin{itemize}
 %   \item the eigenvectors $x_k$ of $J_{m,n}$ associated to $\lambda_k^{(n)}$ (eigenvalue of  $C_n$), with $k$ not multiple of $\frac{n}{m}$ are independently linear with the eigenvectors that come from a root of $\pi_{J_{m,n}}$, %esto vale para cualquier m incluso impar
  %  \item the eigenvalues of $J_{m,n}$ common to $C_m$ given in Corollary \ref{cor:vapsDeCmEnJmn} are not roots of $\pi_{J_{m,n}}$. %esto creemos que no, al menos con el mismo argumento
%\end{itemize}

In order to prove the first item we will prove that $\langle x_k, v \rangle=0$ for all $v$ eigenvector that come from a root of $\pi_{J_{m,n}}$. Remember that any eigenvector $v$ can be expressed as $Pw$ with $w$ eigenvector of $A(J_{m,n}/{\pi})$ and $P$ the characteristic matrix of $\pi$, then, it will be enough to prove that  $\langle x_k, P_i \rangle=0$ for all $P_i$ column of $P$.

By Proposition~\ref{prop:common-no-multiple} we have that $x_k=\begin{pmatrix}
y_k\\ 
0
\end{pmatrix}$ with $y_k$ eigenvector of the cycle $C_n$. For convenience, we enumerate the coordinates of $x_k$ from $0$ to $n$.

We have that $\langle x_k,P_1\rangle=0$, because $P_1=\begin{pmatrix}
\textbf{0}\\ 
1
\end{pmatrix}$ is the column that corresponds to $S_1$ (the center vertex).\\

In order to compute $\langle x_k,P_2\rangle$ we shall observe that $(P_2)_i=\begin{cases} 1 &\text{if  }0\leq i\leq n-1, \; i \equiv 0 \ (\text{mod } m),\\
0 &\text{otherwise,}
\end{cases}$\\ then,

\[\langle x_k,P_2\rangle=\sum_{\tiny\begin{matrix}
i\equiv 0 \ (\text{mod} m)\\ 
0\leq i\leq n -1
\end{matrix}}\rho_k^{i}=0\]
 by Lemma \ref{lem:sumatoria}.

For columns $P_j$ with $3\leq j \leq \frac{m}{2}+2$ we have that
$(P_j)_i=\begin{cases} 1 &\text{if  }0\leq i\leq n-1 \text{ and }\\
&i \equiv j-2 \ (\text{mod } m) \text{ or } i \equiv 2-j \ (\text{mod } m),\\
0 &\text{otherwise.}
\end{cases}$\\
 
Then, 
\[\langle x_k,P_j\rangle=\sum_{\tiny\begin{matrix}
0\leq i\leq n-1\\
i\equiv j-2 (\text{mod }m)
\end{matrix}}(\rho_k)^{i}
+
\sum_{\tiny\begin{matrix}
0\leq i\leq n-1\\
i\equiv 2-j \ (\text{mod }m) 
\end{matrix}}(\rho_k)^{i}
=
\sum_{i=0}^{\frac{n}{m}-1}(\rho_k)^{j-2+im}
+
\sum_{i=0}^{\frac{n}{m}-1}(\rho_k)^{2-j+im}
=
\]
\[
\left((\rho_k)^{j-2}+(\rho_k)^{2-j}\right)
\sum_{i=0}^{\frac{n}{m}-1}(\rho_k)^{im}
=\left((\rho_k)^{j-2}+(\rho_k)^{2-j}\right)
\sum_{\scriptsize\begin{matrix} 
0\leq i\leq n-1 \\ 
i\equiv0 (\text{mod }m)
\end{matrix}} (\rho_k)^i.
\]
Given that $k$ is not a multiple of $\frac{n}{m}$ by Lemma \ref{lem:sumatoria} we have that \[\sum_{\scriptsize\begin{matrix} 
0\leq i\leq n-1 \\ 
i\equiv0 (\text{mod }m)
\end{matrix}} (\rho_k)^i=0,\] then, $\langle x_k,P_j\rangle=0$.\\

%The orthogonality argument in Section~\ref{sec:common} shows that eigenvectors  corresponding to $\lambda_k^{(n)}$ (with $k \not\equiv 0 \bmod n/m$) are orthogonal to all  lifted eigenvectors coming from $Q_{m,n}$.  Thus these eigenvalues do not arise from the quotient polynomial.

%Similarly, the eigenvalues of $C_m$ that appear in $J_{m,n}$ by interlacing  (Corollary~\ref{cor:vapsDeCmEnJmn}) cannot be roots of $\pi_{J_{m,n}}(x)$, since the corresponding $\phi(P_l,x)$ polynomials have disjoint spectra, where $P_l$ denotes the simple path.

For the second item we begin to prove case $m=4$ and then prove the general case.\\

For $m =4$ we have that $Sp(P_{1})= \left[0\right]$ which is {exactly} the set of common eigenvalues to $J_{4,n}$ and $C_4$ mentioned in Corollary \ref{cor:vapsDeCmEnJmn}. It suffices to check that $0$ is not a root of $-2\left(x^2-\frac{n}{4} \right) \phi(P_0,x)$, indeed,
\[-2\left(0^2-\frac{n}{4} \right) \phi(P_0,0)= \frac{n}{2}.\]

Finally, we obtain that \[\phi(J_{4,n},x)=\left(\prod_{\tiny\begin{matrix}
1\leq k\leq n -1\\ k \not\equiv 0 \ (\text{mod} \frac{n}{4})
\end{matrix}}
%_{
%\scriptsize \begin{matrix}
%k=1, \hdots, n-1,\\
% k \text{ not multiple of }\frac{n}{4}
%\end{matrix}}
\left(x- \lambda_k^{(n)}\right)\right) \cdot x \cdot \pi_{J_{m,n}}(x).\]

For $m \geq 6$ we have that $Sp(P_{\frac{m}{2} -1})= \left[2\cos\left(\frac{2\pi j}{m}\right): j=1, \hdots, \frac{m}{2}-1\right]$ which is {exactly} the set of common eigenvalues to $J_{m,n}$ and $C_m$ mentioned in Corollary \ref{cor:vapsDeCmEnJmn}.\\ %\textcolor{red}{hay que explicar porqué 2 y -2 no son vaps de los que están en $J_{m,n}$}.\\

Let us show that the eigenvalues of $P_{\frac{m}{2} -1}$ are not eigenvalues of $P_{\frac{m}{2} -2}$, and so, they are not roots of $\pi_{J_{m,n}}$.\\
Indeed, \[Sp(P_{\frac{m}{2}-2})=\left[2\cos\left(\frac{2\pi j}{m-2}\right): j=1, \hdots, \frac{m}{2}-2 \right]\] and $2 \cos\left(\frac{2\pi j}{m}\right)=2 \cos\left(\frac{2\pi i}{m-2}\right)$ for some $i=1, \hdots, \frac{m}{2}-1$, $j=1, \hdots, \frac{m}{2}-2$ if, and only if, $j-i=\frac{2}{m}j \in \mathbb{Z}$, then, $j$ should be a multiple of $\frac{m}{2}$ which is a contradiction given the rank where $j$ varies.\\

Since $J_{m,n}$ has order $n+1$ and the previous two families contribute 
$n-1 - (m-2)$ eigenvalues, exactly $\frac{m}{2}+2$ eigenvalues remain.  
These correspond to the degree of $\pi_{J_{m,n}}(x)$, proving the claim.
\end{proof}

Combining all components yields the following factorization.

\begin{cor}\label{cor:factor-Jmn}
Let $m\ge 4$ be even. Then
\[
\phi(J_{m,n},x)
=
\left(
\prod_{\substack{1\le k\le n-1\\ k\not\equiv 0 \,(\mathrm{mod}\, n/m)}} 
(x-\lambda_k^{(n)})
\right)
\left(
\prod_{i=1}^{\frac{m-2}{2}}
(x-\lambda_i^{(m)})
\right)
\pi_{J_{m,n}}(x).
\]
Equivalently, using $\phi(C_n,x)/\phi(C_m,x)$,
\[
\phi(J_{m,n},x)
=
\frac{\phi(C_n,x)}{\phi(C_m,x)}
\left(
\prod_{i=1}^{\frac{m-2}{2}}
(x-\lambda_i^{(m)})
\right)
\pi_{J_{m,n}}(x).
\]

\end{cor}

\subsection{The case of odd {$m$}}
\label{subsec:odd}
The orthogonality argument used in Proposition~\ref{prop:remaining-eigenvalues} does not depend on the parity of $m$, 
and therefore the roots of $\pi_{J_{m,n}}(x)$ still correspond to the missing eigenvalues of $J_{m,n}$ when $m$ is odd.  
However, unlike the even case, the expressions for $\pi_{J_{m,n}}(x)$ cannot be rewritten cleanly 
in terms of $\phi(C_m,x)$ or path polynomials indexed by $m/2$; the underlying symmetry breaks when $m$ is odd.  
Thus a closed factorization analogous to Corollary~\ref{cor:factor-Jmn} is not available in the odd case.

\section{Spectrum and energy of $J_{2,n}$, $J_{4,n}$ and $J_{6,n}$}

In this section we determine explicitly the spectrum and the energy of the Jahangir graphs
$J_{2,n}$, $J_{4,n}$ and $J_{6,n}$.  
In each case, the eigenvalues that $J_{m,n}$ inherits from the cycle $C_n$ are described by
Section~\ref{sec:common}, while the remaining eigenvalues are obtained as the roots of the
quotient polynomial $\pi_{J_{m,n}}(x)$ from Section~\ref{sec:factorization}.

\subsection{The graph $J_{2,n}$}

From Section~\ref{subsec:factor-even} we have
\[
\phi(J_{2,n},x)
=
\frac{\phi(C_n,x)}{(x-2)(x+2)}\,\pi_{J_{2,n}}(x),
\]
and the roots of $\pi_{J_{2,n}}(x)$ are
\[
\left\{\sqrt{\tfrac{n}{2}+4},\,0,\,-\sqrt{\tfrac{n}{2}+4}\right\}.
\]
Thus we obtain the following description.

\begin{teo}\label{teo:J2n}
The spectrum of $J_{2,n}$ is
\[
Sp(J_{2,n})
=
Sp(C_n)-[2,-2]
\;\sqcup\;
\Bigl[\sqrt{\tfrac{n}{2}+4},\,0,\,-\sqrt{\tfrac{n}{2}+4}\Bigr].
\]

The energy of $J_{2,n}$ is
\[
\mathcal{E}(J_{2,n})
=
\mathcal{E}(C_n) - 4 + \sqrt{2n+16}.
\]
\end{teo}

\begin{proof}
The factor $\phi(C_n,x)/\bigl((x-2)(x+2)\bigr)$ removes the eigenvalues $2$ and $-2$ from $Sp(C_n)$, 
and the factor $\pi_{J_{2,n}}(x)$ contributes the three eigenvalues
$\sqrt{\tfrac{n}{2}+4}$, $0$ and $-\sqrt{\tfrac{n}{2}+4}$.  
This gives the claimed spectral decomposition.

For the energy we use
\[
\mathcal{E}(J_{2,n})
=
\mathcal{E}(C_n)
-
(|2|+|-2|)
+
\Bigl|\sqrt{\tfrac{n}{2}+4}\Bigr|
+
|0|
+
\Bigl|-\sqrt{\tfrac{n}{2}+4}\Bigr|
=
\mathcal{E}(C_n) - 4 + 2\sqrt{\tfrac{n}{2}+4},
\]
and
$2\sqrt{\tfrac{n}{2}+4} = \sqrt{2n+16}$.
\end{proof}

\subsection{The graph $J_{4,n}$}

We now turn to the case $m=4$.  
Here the quotient polynomial $\pi_{J_{4,n}}(x)$ is quartic and its roots can be expressed in closed form.

\begin{teo}\label{teo:J4n}
The spectrum of $J_{4,n}$ is
\[
Sp(J_{4,n})
=
Sp(C_n)-[2,0,-2]
\;\sqcup\;
[\theta,-\theta,\rho,-\rho],
\]
where
\[
\rho=\sqrt{\frac{n}{8}+2+\sqrt{\frac{n^2}{64}+4}}
\qquad\text{and}\qquad
\theta=\sqrt{\frac{n}{8}+2-\sqrt{\frac{n^2}{64}+4}}.
\]

The energy of $J_{4,n}$ is
\[
\mathcal{E}(J_{4,n})
=
\mathcal{E}(C_n) - 4 + 2\rho + 2\theta.
\]
\end{teo}

\begin{proof}
From Proposition~\ref{prop:quotient-polynomial} we have
\[
\pi_{J_{4,n}}(x)
=
\bigl((x^2-\tfrac{n}{4})x - 2x\bigr)\phi(P_1,x)
-
2\bigl(x^2-\tfrac{n}{4}\bigr)\phi(P_0,x).
\]
Using $\phi(P_1,x)=x$ and $\phi(P_0,x)=1$, we obtain
\[
\pi_{J_{4,n}}(x)
=
x^4 - \Bigl(4+\frac{n}{4}\Bigr)x^2 + \frac{n}{2}.
\]
Setting $y=x^2$, this becomes a quadratic in $y$:
\[
y^2 - \Bigl(4+\frac{n}{4}\Bigr)y + \frac{n}{2} = 0,
\]
whose solutions are
\[
y = \frac{n}{8}+2 \pm \sqrt{\frac{n^2}{64}+4}.
\]
Thus the roots of $\pi_{J_{4,n}}(x)$ are
\[
\pm\rho,\quad \pm\theta,
\]
with $\rho,\theta>0$ as defined above.

By Remark~\ref{obs:espectroCnmenosCm}, the eigenvalues $2$, $0$ and $-2$ of $C_n$ do not 
survive in $Sp(J_{4,n})$, and are replaced by $\pm\rho,\pm\theta$.  
This yields the asserted description of the spectrum.

For the energy, we remove $2$, $0$, $-2$ from the cycle and add $\pm\rho,\pm\theta$, obtaining
\[
\mathcal{E}(J_{4,n})
=
\mathcal{E}(C_n) - (|2|+|0|+|-2|) + (|\rho|+|-\rho|+|\theta|+|-\theta|)
=
\mathcal{E}(C_n) - 4 + 2\rho + 2\theta.
\]
\end{proof}

\subsection{The graph $J_{6,n}$}

Finally we treat the case $m=6$, where the quotient polynomial splits into a linear factor and a quartic factor that can again be solved explicitly.

\begin{teo}\label{teo:J6n}
The spectrum of $J_{6,n}$ is
\[
Sp(J_{6,n})
=
Sp(C_n)-[-2,-1,1,2]
\;\sqcup\;
[-\rho,-\theta,0,\theta,\rho],
\]
where
\[
\rho
=
\sqrt{\frac{n}{12}+\frac{5}{2}+\sqrt{\frac{n^2}{144}-\frac{n}{12}+\frac{9}{4}}}
\quad\text{and}\quad
\theta
=
\sqrt{\frac{n}{12}+\frac{5}{2}-\sqrt{\frac{n^2}{144}-\frac{n}{12}+\frac{9}{4}}}.
\]

The energy of $J_{6,n}$ is
\[
\mathcal{E}(J_{6,n})
=
\mathcal{E}(C_n) - 6 + 2\rho + 2\theta.
\]
\end{teo}

\begin{proof}
From Proposition~\ref{prop:quotient-polynomial} we have
\[
\pi_{J_{6,n}}(x)
=
\bigl(x^3-(\tfrac{n}{6}+4)x\bigr)\phi(P_2,x)
+
\frac{n}{3}\,\phi(P_1,x).
\]
Since $\phi(P_2,x)=x^2-1$ and $\phi(P_1,x)=x$, we obtain
\[
\pi_{J_{6,n}}(x)
=
x\Bigl(x^4 - \Bigl(5+\frac{n}{6}\Bigr)x^2 + \frac{n}{2} + 4\Bigr).
\]
Again setting $y=x^2$, the quartic factor corresponds to
\[
y^2 - \Bigl(5+\frac{n}{6}\Bigr)y + \Bigl(\frac{n}{2}+4\Bigr) = 0,
\]
whose solutions are
\[
y
=
\frac{n}{12}+\frac{5}{2}
\pm
\sqrt{\frac{n^2}{144}-\frac{n}{12}+\frac{9}{4}}.
\]
Thus the roots of $\pi_{J_{6,n}}(x)$ are
\[
0,\quad \pm\rho,\quad \pm\theta,
\]
with $\rho,\theta>0$ as stated.

By Remark~\ref{obs:espectroCnmenosCm}, the eigenvalues $-2,-1,1,2$ of $C_n$ are replaced by
$0,\pm\rho,\pm\theta$ in $Sp(J_{6,n})$, which gives the claimed spectral decomposition.

For the energy, we remove $-2,-1,1,2$ and add $0,\pm\rho,\pm\theta$:
\[
\mathcal{E}(J_{6,n})
=
\mathcal{E}(C_n)
-
\bigl(|-2|+|-1|+|1|+|2|\bigr)
+
\bigl(|0|+|\rho|+|-\rho|+|\theta|+|-\theta|\bigr)
=
\mathcal{E}(C_n) - 6 + 2\rho + 2\theta.
\]
\end{proof}

\section{Comparison of energies}

Once closed formulas for the energies of $C_n$, $W_n$, and the Jahangir graphs $J_{m,n}$ (with $m=2,4,6$) are established, it is natural to compare them. 
Since $C_n$ is an induced subgraph of every $J_{m,n}$, Theorem~4.19 in \cite{Li} implies
\[
\mathcal{E}(C_n)\le \mathcal{E}(J_{m,n})
\qquad\text{for all }m\mid n.
\]
We now compare the energies of the three families studied in detail.

\begin{prop}\label{prop:EJ6nEJ4n}
Let $n$ be a positive integer divisible by $12$. Then
\[
\mathcal{E}(J_{6,n})\le \mathcal{E}(J_{4,n}).
\]
\end{prop}

\begin{proof}
By Theorems~\ref{teo:J4n} and \ref{teo:J6n},
\[
\mathcal{E}(J_{4,n})
=
\mathcal{E}(C_n)-4+2\rho'+2\theta'
=
\mathcal{E}(J_{6,n})+2-2\rho-2\theta+2\rho'+2\theta',
\]
where
\[
\rho=\sqrt{\frac{n}{12}+\frac{5}{2}+\sqrt{\frac{n^2}{144}-\frac{n}{12}+\frac94}}, 
\qquad
\theta=\sqrt{\frac{n}{12}+\frac{5}{2}-\sqrt{\frac{n^2}{144}-\frac{n}{12}+\frac94}},
\]
\[
\rho'=\sqrt{\frac{n}{8}+2+\sqrt{\frac{n^2}{64}+4}}, 
\qquad
\theta'=\sqrt{\frac{n}{8}+2-\sqrt{\frac{n^2}{64}+4}}.
\]

Thus the claim is equivalent to
\[
  \rho+\theta \;\le\; 1+\rho'+\theta'.
\]

%\textcolor{green}{ésto es lo que iría, chat gpt me dio ok con las cuentas}\\
%\textcolor{blue}{Squaring both terms we obtain the following inequality:\\
%$\frac{n}{4}+5+2\sqrt{\frac{n}{2}}+2\sqrt{\frac{n}{8}+2+\sqrt{\frac{n^2}{64}+4}}+2\sqrt{\frac{n}{8}+2-\sqrt{\frac{n^2}{64}+4}}\geq\frac{n}{6}+5+2\sqrt{\frac{n}{2}+4}$\\
%Since $\frac{n}{4}\geq\frac{n}{6}, 2\sqrt{\frac{n}{8}+2-\sqrt{\frac{n^2}{64}+4}}\geq0$ for all $n\in\mathbb{N}$ and $2\sqrt{\frac{n}{2}}+2\sqrt{\frac{n}{8}+2+\sqrt{\frac{n^2}{64}+4}}\geq2\sqrt{\frac{n}{2}+4}$ for all $n\geq12$, we can assume the inequality helds for any $n$ positive integer divisible by 12.}\\

%\textcolor{green}{ésto es lo que estaba antes que chat gpt me dijo que no se cumplía:}\\
We will actually show that $\rho+\theta \le \rho'+\theta',$ a direct computation shows that 
\[
\rho+\theta \le \rho'+\theta'
\quad 
\Longleftrightarrow \quad 12- n\le \sqrt{288n}-\sqrt{96n-144}.
\]

The right–hand side is nonnegative for all $n$, hence the inequality holds for every $n$ satisfying the hypothesis.

%\Longleftrightarrow \quad \frac{n}{6} + 5 +\sqrt{\frac{2n}{3}+1} \le\frac{n}{4} + 4 +\sqrt{2n}
\end{proof}

\begin{prop}
Let $n\ge 32$ be a positive integer divisible by $4$. Then
\[
\mathcal{E}(J_{4,n})\le \mathcal{E}(J_{2,n}).
\]
\end{prop}

\begin{proof}
From Theorems~\ref{teo:J2n} and \ref{teo:J4n},
\[
\mathcal{E}(J_{2,n})
=
\mathcal{E}(C_n)-4+\sqrt{2n+16}
=
\mathcal{E}(J_{4,n})+\sqrt{2n+16}-2\rho-2\theta,
\]
where
\[
\rho=\sqrt{\frac{n}{8}+2+\sqrt{\frac{n^2}{64}+4}},
\qquad
\theta=\sqrt{\frac{n}{8}+2-\sqrt{\frac{n^2}{64}+4}}.
\]

Thus it suffices to prove
\[
2\rho+2\theta \le \sqrt{2n+16}.
\]

By squaring, this is equivalent to
\[
32n \le n^2,
\]
which holds for all $n\ge 32$.
\end{proof}

\begin{obs}\label{obs:falla}
The preceding argument shows that for multiples of $4$ with $n<32$ the inequality reverses, i.e.,
\[
\mathcal{E}(J_{2,n})\le \mathcal{E}(J_{4,n}).
\]
\end{obs}

\begin{prop}
Let $n$ be even. Then
\[
\mathcal{E}(J_{2,n})\le \mathcal{E}(W_n).
\]
\end{prop}

\begin{proof}
From Proposition~\ref{prop:energiaCicloRueda} and Theorem~\ref{teo:J2n},
\[
\mathcal{E}(W_n)
=
\mathcal{E}(C_n)+2(\sqrt{n+1}-1)
=
\mathcal{E}(J_{2,n})+2-\sqrt{2n+16}+2\sqrt{n+1}.
\]
Thus we need to show
\[
2-\sqrt{2n+16}+2\sqrt{n+1}\ge 0.
\]

Equivalently,
\[
\sqrt{2n+16}\le 2+2\sqrt{n+1}
\quad \Longleftrightarrow \quad
8\le 8\sqrt{n+1}+2n,
\]
which holds for all $n\ge 1$.
\end{proof}

\begin{cor}\label{energycomp}
Let $n\ge 36$ be a multiple of $12$. Then
\[
\mathcal{E}(C_n)\le \mathcal{E}(J_{6,n})\le \mathcal{E}(J_{4,n})\le
\mathcal{E}(J_{2,n})\le \mathcal{E}(W_n).
\]
\end{cor}

The individual inequalities cannot be deduced by transitivity alone due to the divisibility constraints on $n$, hence each has been treated separately.

\begin{prop}
Let $n$ be a positive integer divisible by $6$. Then
\[
\mathcal{E}(J_{6,n})\le \mathcal{E}(J_{2,n}).
\]
\end{prop}

\begin{proof}
From Theorems~\ref{teo:J2n} and \ref{teo:J6n},
\[
\mathcal{E}(J_{2,n})
=
\mathcal{E}(J_{6,n})+2-2\rho-2\theta+\sqrt{2n+16},
\]
with $\rho,\theta$ as in Theorem~\ref{teo:J6n}.  

Thus it suffices to show
\[
2\rho+2\theta \le \sqrt{2n+16}+2.
\]

Squaring yields the equivalent condition
\[
0\le \tfrac{n}{3},
\]
which holds for all $n\ge 1$.
\end{proof}

\begin{cor}
If $n$ is a positive integer divisible by $6$, then
\[
\mathcal{E}(J_{6,n})\le \mathcal{E}(W_n).
\]
\end{cor}

\begin{prop}
If $n$ is a positive integer divisible by $4$, then
\[
\mathcal{E}(J_{4,n})\le \mathcal{E}(W_n).
\]
\end{prop}

\begin{proof}
By Theorem~\ref{teo:J4n} and Proposition~\ref{prop:energiaCicloRueda},
\[
\mathcal{E}(W_n)
=
\mathcal{E}(J_{4,n})-2\rho-2\theta+2(\sqrt{n+1}+1),
\]
where $\rho,\theta$ are as in Theorem~\ref{teo:J4n}.  
Thus we require
\[
\rho+\theta\le \sqrt{n+1}+1.
\]

Squaring yields
\[8+4\sqrt{2n}\le 3n + 8 \sqrt{n+1},\]
which holds for all $n\ge 4$.
%\[0\le \frac{3n}{4}-2+2\left(\sqrt{n+1}-\sqrt{\tfrac{n}{2}}\right),\]

%\textcolor{blue}{creo que acá no hay que darle bola al claude}
\end{proof}

\section{Conclusions}

Corollary~\ref{energycomp} shows a monotone chain of energies for $J_{2,n}$, $J_{4,n}$, and $J_{6,n}$ when $n$ satisfies suitable divisibility conditions. We have shown that, when edges are removed from the wheel $W_n$ following the structured pattern that produces the Jahangir graphs, the energy decreases in a predictable way. Moreover, within this restricted class of deletions, the more spokes removed, the lower the energy of the resulting graph.
  
We finalize this paper suggesting some new problems.

 \textbf{Problem 1} The result exposed in Corollary~\ref{energycomp} suggests that the number of cycles in the graph plays a central role in determining its energy. However, Remark~\ref{obs:falla} shows that this behavior fails for small values of $n$, and understanding the precise mechanism behind this anomaly remains an interesting problem.

 \textbf{Problem 2} A natural question is whether such monotonicity persists beyond this highly symmetric deletion pattern.  
    Does every deletion of $m$ evenly spaced spokes decrease the energy more than the deletion of $m-1$ such spokes?
 This motivates the broader conjectural inequality:
    \[
    \mathcal{E}(C_n)\le \mathcal{E}(J_{6,n})\le \mathcal{E}(J_{5,n})\le
    \mathcal{E}(J_{4,n})\le \mathcal{E}(J_{3,n})\le
    \mathcal{E}(J_{2,n})\le \mathcal{E}(W_n),
    \]
    valid whenever the graphs $J_{m,n}$ are defined.  
    Establishing or refuting this full ordering remains an open problem.
    
    \textbf{Problem 3} We wonder if Proposition \ref{prop:remaining-eigenvalues} remains valid if $m$ is odd, given that the first item is already proved.

\end{document}